\documentclass[a4paper, 11pt, twoside, notitlepage]{amsart}

\usepackage{amsmath, amssymb, amsthm, amscd}
\usepackage{mathtools}
\usepackage{mathrsfs}

\usepackage{graphicx, xcolor}
\usepackage{tikz}

\usepackage[ocgcolorlinks, linkcolor=blue]{hyperref}
\usepackage{url}
\usepackage[shortlabels]{enumitem}
\usepackage{comment}
\usepackage{verbatim}

\usepackage{bm}
\usepackage{bbm}
\usepackage{dsfont}
\usepackage{relsize}
\usepackage{esint}
\usepackage[normalem]{ulem}

\usepackage[utf8]{inputenc}

\newcommand\redsout{\bgroup\markoverwith{\textcolor{red}{\rule[0.5ex]{2pt}{0.4pt}}}\ULon}

\usepackage{tikz}
\usepackage{dsfont}
\usepackage{relsize}
\usepackage{url}
\usepackage{xcolor}
\usepackage{graphicx}
\usepackage{mathrsfs}
\usepackage[shortlabels]{enumitem}
\usepackage{lineno}
\usepackage{amsmath}
\usepackage{enumitem}
\usepackage{amsthm} 
\usepackage{verbatim}
\usepackage{dsfont}
\numberwithin{equation}{section}

\DeclareMathOperator{\tr}{tr}

\allowdisplaybreaks

\mathtoolsset{showonlyrefs}

\graphicspath{{images/}}

\newtheorem{theorem}{Theorem}[section]
\newtheorem{lemma}[theorem]{Lemma}
\newtheorem{remark}[theorem]{Remark}

\newtheorem{proposition}[theorem]{Proposition}
\newtheorem*{lemma*}{Lemma}
\theoremstyle{definition}
\newtheorem{definition}[theorem]{Definition}
\newtheorem{corollary}[theorem]{Corollary}

\newtheorem{assumption}[theorem]{Assumption}
\newtheorem{example}[theorem]{Example}

\title[Fully nonlinear elliptic equations]{An inverse source problem for a fully nonlinear elliptic equation}

\author[C.-L. Lin]{Ching-Lung Lin}
\address{Department of Mathematics, National Cheng Kung University, Tainan 701, Taiwan.}
\email{cllin2@mail.ncku.edu.tw}

\author[Y.-H. Lin]{Yi-Hsuan Lin}
\address{Department of Applied Mathematics, National Yang Ming Chiao Tung University, Hsinchu, Taiwan \& Fakult\"at f\"ur Mathematik, University of Duisburg-Essen, Essen, Germany}
\curraddr{}
\email{yihsuanlin3@gmail.com}

\author[J.-N. Wang]{Jenn-Nan Wang}
\address{Institute of Applied Mathematical Sciences, National Taiwan University, Taipei 106, Taiwan}
\curraddr{}
\email{jnwang@math.ntu.edu.tw}

\keywords{Inverse problems, fully nonlinear, concave, anisotropic Calder\'on problem.}
\subjclass[2020]{35R30, 35J25, 35J60, 35J96}

\newcommand{\R}{{\mathbb R}}

\newcommand{\cS}{{S}}

\newcommand{\eps}{\epsilon}

\newcommand {\p} {\partial}

\begin{document}
	
	\begin{abstract}
		We study an inverse source problem for fully nonlinear elliptic equations of the form
		\[
		F(D^2u)=f \quad \text{in } \Omega.
		\]
		The question is whether the source term can be recovered from the Dirichlet-to-Neumann map. In two dimensions, the first linearization does not immediately give uniqueness: it leaves a natural conformal ambiguity in the linearized coefficients. For homogeneous nonlinearities $F$ with injective differential $DF$, we show that this ambiguity has a precise meaning at the level of the equation itself, namely that the source is determined up to an explicit scalar factor.
		
		The main point of the paper is to show how this remaining factor can be removed. We use the second linearization to extract information which is invisible at first order, and combine it with an algebraic nondegeneracy condition on the nonlinearity. Under this condition, the residual ambiguity is forced to be trivial, and the Dirichlet-to-Neumann map uniquely determines the source. The result applies, in particular, to homogeneous admissible Hessian equations of Monge--Amp\`ere type and related examples.
	\end{abstract}
	
	\maketitle

	\tableofcontents
	
	\section{Introduction}\label{sec: introduction}
	
	In this work, we study the inverse source problem for a class of fully nonlinear elliptic equations. From a physical point of view, the source term $f$ models internal forcing mechanisms of the medium, such as distributed loads, heat generation, or external fields acting inside the domain, while the state variable $u$ describes the resulting equilibrium configuration or potential. By prescribing boundary values of $u$ and measuring the induced boundary flux $\partial_\nu u$, one performs noninvasive boundary experiments on the system. The inverse problem is then to determine whether these boundary measurements uniquely recover the internal source $f$.
	
	The mathematical model is given by the boundary value problem
	\begin{equation}\label{eq: main}
		\begin{cases}
			F(D^2u)=f & \text{in } \Omega, \\
			u=\varphi & \text{on } \partial\Omega,
		\end{cases}
	\end{equation}
	where $\Omega\subset \R^2$ is a bounded domain with $C^\infty$ boundary $\partial\Omega$. In this paper, we work in the admissible setting for Hessian equations as in \cite{CNS_nonlinear_Hessian}. More precisely, let $\Gamma\subset \R^2$ be an open convex cone with vertex at the origin, symmetric with respect to permutations of the coordinates, and containing the positive cone
	\begin{equation}\label{eq: Gamma plus}
		\Gamma_+:=\left\{(\lambda_1,\lambda_2)\in \R^2:\,  \lambda_1>0,\ \lambda_2>0\right\}.
	\end{equation}
	We assume that $F(M)=\widetilde f(\lambda(M))$, where $\lambda(M)=(\lambda_1(M),\lambda_2(M))$ denotes the eigenvalues of $M$, and $\widetilde f\in C^\infty(\Gamma)\cap C^0(\overline{\Gamma})$ is symmetric. Let $\cS$ denote the vector space of all $2\times 2$ real symmetric matrices. We write
	\[
	{\mathcal A}_\Gamma:=\{M\in \cS:\, \lambda(M)\in\Gamma\}
	\]
	for the admissible set determined by $\Gamma$.
	
	Following \cite{CNS_nonlinear_Hessian}, we impose the following admissibility assumptions on $\widetilde f$.
	
	\begin{assumption}\label{assump1}~
		\begin{enumerate}[(i)]
			\item\label{item 1 - elliptic} $\widetilde f$ is elliptic in $\Gamma$, namely
			\[
			\frac{\partial \widetilde f}{\partial \lambda_i}>0
			\quad \text{in }\, \Gamma,\quad i=1,2;
			\]
			
			\item\label{item 2 -concavity} $\widetilde f$ is concave in $\Gamma$;
			
			\item\label{item 3 - boundary} the boundary-side condition
			\begin{equation}\label{mini}
				\limsup_{\lambda\to\lambda_0}\widetilde f(\lambda)\le \overline f_0
				\quad \text{for all } \lambda_0\in \partial\Gamma
			\end{equation}
			holds for the prescribed source $f>0$, where
			\[
			f_0:=\min_{\overline{\Omega}}f
			\quad \text{and} \quad
			0<\overline f_0<f_0;
			\]
			
			\item\label{item 4 - cone 1} for every compact set $K\subset \Gamma$ and every $C>0$, there exists $R=R(K,C)>0$ such that
			\[
			\widetilde f(\lambda_1,\lambda_2+R)\ge C
			\quad \text{for all } (\lambda_1,\lambda_2)\in K;
			\]
			
			\item\label{item 5 - cone 2} for every compact set $K\subset \Gamma$ and every $C>0$, there exists $R=R(K,C)>0$ such that
			\[
			\widetilde f(R\lambda)\ge C
			\quad \text{for all } \lambda\in K.
			\]
		\end{enumerate}
	\end{assumption}
	
	We also assume an admissibility condition on the boundary of the domain: there exists $T>0$ such that, at every point of $\partial\Omega$, if $\kappa_1$ denotes the principal curvature of $\partial\Omega$ with respect to the interior normal, then
	\begin{equation}\label{eq: CNS domain intro}
		(\kappa_1,T)\in \Gamma.
	\end{equation}
	
	\begin{definition}[Admissible solutions]
		A solution $u\in C^2(\overline{\Omega})$ of \eqref{eq: main} with $u=\varphi$ on $\partial\Omega$ is called \emph{admissible} if $\lambda(D^2u(x))\in \Gamma$ for all $x\in \overline{\Omega}$.
	\end{definition}
	
	For an admissible matrix $M$, the ellipticity condition $\partial_{\lambda_i}\widetilde f>0$ implies that the matrix $\big(F_{ab}(M)\big)$ is positive-definite, where
	\begin{equation}\label{eq: F_first derivative}
		F_{ab}(M):=\frac{\partial F}{\partial M_{ab}}(M), \quad a,b=1,2.
	\end{equation}
	In particular, for every admissible solution $u$, the first linearized operator $F_{ab}(D^2u)\partial_{ab}$ is elliptic.
	
	Under the above assumptions, the direct problem \eqref{eq: main} is well posed in the admissible class by \cite[Theorem~2]{CNS_nonlinear_Hessian}. We shall recall in Section~\ref{sec: prel} the precise well-posedness statement needed in this paper. Thanks to the well-posedness of \eqref{eq: main}, one can define the associated Dirichlet-to-Neumann (DN) map
	\begin{equation}\label{eq: DN map}
		\Lambda_f:C^\infty(\partial\Omega)\to C^\infty(\partial\Omega),
		\quad
		\varphi\mapsto \partial_\nu u_\varphi\big|_{\partial\Omega},
	\end{equation}
	where $u_\varphi\in C^\infty(\overline{\Omega})$ is the unique admissible solution to \eqref{eq: main}. The inverse problem considered in this paper is the following:
	\begin{enumerate}[\textbf{(IP)}]
		\item \label{Q:IP} Can one uniquely determine the source function $f$ in \eqref{eq: main} from the knowledge of the DN map?
	\end{enumerate}
	
	Equations of the form \eqref{eq: main} arise, for instance, in nonlinear elasticity, where the unknown $u$ represents the displacement of the material and the operator $F(D^2u)$ describes a nonlinear constitutive law relating strain or curvature to internal stress. In this setting, the source term $f$ corresponds to body forces acting inside the material, and the inverse source problem amounts to recovering hidden internal forces from boundary displacement and traction measurements.
	
	Inverse source problems have been extensively studied for linear elliptic equations, but remain far less understood in the fully nonlinear setting. As a notable example, when $F(M)=\det M$, corresponding to the Monge-Amp\`ere equation, an analogue of the inverse problem \ref{Q:IP} was recently resolved in~\cite{LL2025IP_Monge_Ampere}. The present work is motivated by the question of whether similar uniqueness results continue to hold for a wider class of admissible fully nonlinear elliptic equations. We first obtain a determination result up to an explicit scalar factor, and then impose an additional non-proportionality condition to remove this remaining freedom and obtain full uniqueness.
	
	To study the inverse problem, we impose the following additional structural assumptions on $F$. For an admissible matrix $M=(M_{ab})_{1\le a,b\le 2}$, we write
	\begin{equation}\label{eq: F_second derivative}
		F_{ab,k\ell}(M):=\frac{\partial^2 F}{\partial M_{ab}\partial M_{k\ell}}(M), \quad a,b,k,\ell=1,2.
	\end{equation}
	
	\begin{assumption}\label{ass: homogeneous nondegeneracy}
		In addition to Assumption~\ref{assump1}, we further assume that:
		\begin{enumerate}[{\rm(i)}]
			\item\label{item ass_positive_homo} $F$ is positively homogeneous of degree $m$ with $0<m\neq 1$, that is,
			\begin{equation}
				F(tM)=t^mF(M)
				\quad \text{for all admissible } M \text{ and all } t>0;
			\end{equation}
			
			\item\label{item ass_DF_inj} $DF$ is injective on the admissible class, namely
			\begin{equation}
				DF(M_1)=DF(M_2)
				\implies
				M_1=M_2
			\end{equation}
			for all admissible matrices $M_1,M_2$;
			
			\item\label{item ass_non-prop} for every admissible matrix $M$, the second derivative tensor $D^2F(M)$ is not proportional to the rank-one tensor $DF(M)\otimes DF(M)$, namely there does not exist any scalar $\gamma\in \mathbb R$ such that
			\begin{equation}
				F_{ab,k\ell}(M)=\gamma\,F_{ab}(M)F_{k\ell}(M)
				\quad \text{for all } 1\le a,b,k,\ell\le 2.
			\end{equation}
		\end{enumerate}
	\end{assumption}
	
	We are now ready to state the main result of this paper.
	
	\begin{theorem}\label{thm: main}
		Let $\Omega\subset \R^2$ be a bounded simply connected domain with $C^\infty$ boundary $\partial\Omega$. Let $F(M)=\widetilde f(\lambda(M))$, where $\widetilde f$ is a smooth symmetric function defined on an open convex cone $\Gamma\subset \R^2$. Assume that \eqref{eq: CNS domain intro} holds, and that $\widetilde f$ satisfies Assumption~\ref{assump1} with respect to each source $f_j$, $j=1,2$. Assume furthermore that $F$ satisfies {\rm Assumption~\ref{ass: homogeneous nondegeneracy}}.
		
		Let $\Lambda_{f_j}$ be the DN map associated with the admissible solution of
		\begin{equation}\label{eq: main j=1,2}
			\begin{cases}
				F(D^2 u^{(j)})=f_j &\text{ in }\Omega, \\
				u^{(j)}=\varphi &\text{ on }\partial\Omega,
			\end{cases}
		\end{equation}
		for $j=1,2$, where $0<f_j\in C^\infty(\overline{\Omega})$ satisfy \eqref{mini}. Suppose that
		\begin{equation}\label{eq: BC for the source}
			f_1 = f_2 \text{ on }\partial\Omega.
		\end{equation}
		Then
		\begin{equation}\label{DN map same}
			\Lambda_{f_1}(\varphi)=\Lambda_{f_2}(\varphi), \quad \text{for all }\varphi \in C^\infty(\partial\Omega),
		\end{equation}
		implies $f_1=f_2$ in $\overline{\Omega}$.
	\end{theorem}
	
	\begin{example}\label{rem: examples homogeneous}
		Let $\Gamma=\Gamma_+$ be the first quadrant of $\R^2$, that is,
		\begin{equation}
			\Gamma=\{(\lambda_1,\lambda_2): \, \lambda_1>0,\ \lambda_2>0\}.
		\end{equation}
		Such a set $\Gamma$ is referred to as \emph{type\,1} in \cite{CNS_nonlinear_Hessian}. Equivalently, the corresponding admissible set is the positive definite cone in the space of symmetric matrices. A basic example is given by
		\begin{equation}
			F(M)=(\det M)^\theta, \quad \theta\in (0,1/2).
		\end{equation}
		Equivalently, $\widetilde f(\lambda_1,\lambda_2)=\lambda_1^\theta\lambda_2^\theta$. It is straightforward to check that $\widetilde f$ satisfies Assumption~\ref{assump1}. Moreover, $F$ is positively homogeneous of degree $2\theta\neq 1$, and
		\begin{equation}
			DF(M)=\theta (\det M)^\theta M^{-1},
		\end{equation}
		which is positive-definite for every $M>0$. Furthermore, when $\theta\in (0,1/2)$, the functional $F(M)=(\det M)^\theta$ is strictly concave on $\Gamma$. In particular, the gradient map $DF$ is injective on $\Gamma$.
		
		To verify the non-proportionality condition, note that
		\begin{equation}
			D^2F(M)[H,K]=\theta (\det M)^\theta\big(\theta \tr(M^{-1}H)\tr(M^{-1}K)-\tr(M^{-1}HM^{-1}K)\big).
		\end{equation}
		Working in a basis in which
		\begin{equation}
			M=\begin{pmatrix}
				\lambda_1 & 0\\
				0 & \lambda_2
			\end{pmatrix},
			\quad \lambda_1,\lambda_2>0,
		\end{equation}
		choose
		\begin{equation}
			H=	\begin{pmatrix}
				\lambda_1 & 0\\
				0 & -\lambda_2
			\end{pmatrix}.
		\end{equation}
		Then $\tr(M^{-1}H)=1-1=0$, so $DF(M):H=0$. On the other hand, $\tr(M^{-1}HM^{-1}H)=2$, which yields
		\begin{equation}
			D^2F(M)[H,H]=-2\theta (\det M)^\theta \neq 0.
		\end{equation}
		If $D^2F(M)$ were proportional to $DF(M)\otimes DF(M)$, then there would exist a scalar $\gamma\in \mathbb R$ such that
		\begin{equation}
			D^2F(M)[H,H]=\gamma (DF(M):H)^2
		\end{equation}
		for every symmetric matrix $H$. Since $DF(M):H=0$ for the above choice of $H$, this would imply $D^2F(M)[H,H]=0$, which is a contradiction. Therefore, $D^2F(M)$ is not proportional to $DF(M)\otimes DF(M)$.
	\end{example}	
	
	\begin{example}\label{CNB-typeII}
		We consider the \emph{type\,2} example given in \cite[p.~264]{CNS_nonlinear_Hessian}. Let
		\begin{equation}
			\widetilde f(\lambda_1,\lambda_2)=\big((\lambda_1+\lambda_2)^2-\tau(\lambda_1-\lambda_2)^2\big)^\theta,
			\quad	0<\tau<1,\quad 0<\theta<1/2.
		\end{equation}
		The corresponding cone is
		\begin{equation}
			\Gamma:=\left\{(\lambda_1,\lambda_2)\in\R^2:\lambda_1+\lambda_2>\sqrt{\tau}\,|\lambda_1-\lambda_2|\right\}.
		\end{equation}
		Then $\Gamma$ is an open convex cone with vertex at the origin, symmetric with respect to permutations of the coordinates, and it contains $\Gamma_+$ as a proper subset.
		
		For $M\in \cS$, set $A(M):=(1-\tau)(\tr M)^2+4\tau\det M$. Then $F(M):=A(M)^\theta$ corresponds to $\widetilde f(\lambda_1,\lambda_2)$, since $A(M)=(\lambda_1+\lambda_2)^2-\tau(\lambda_1-\lambda_2)^2$. Equivalently, if $k=4\tau/(1-\tau)$, then, after factoring out the positive constant $(1-\tau)^\theta$, this corresponds to
		\begin{equation}
			F(D^2u)=\big((\Delta u)^2+k(u_{xx}u_{yy}-u_{xy}^2)\big)^\theta.
		\end{equation}
		
		Let $\mu_1:=\lambda_1+\lambda_2+\sqrt{\tau}(\lambda_1-\lambda_2)$ and $\mu_2:=\lambda_1+\lambda_2-\sqrt{\tau}(\lambda_1-\lambda_2)$. Then $\Gamma$ is transformed into the positive quadrant $\{\mu_1>0,\ \mu_2>0\}$, and $\widetilde f(\lambda_1,\lambda_2)=(\mu_1\mu_2)^\theta$. Hence $\widetilde f$ is elliptic and concave in $\Gamma$, since $0<\theta<1/2$. Moreover, $\widetilde f\to0$ as $\lambda$ approaches $\partial\Gamma$, and the growth conditions in Assumption~\ref{assump1} follow from the positive homogeneity. Thus, $\widetilde f$ satisfies Assumption~\ref{assump1}.
		
		We now verify Assumption~\ref{ass: homogeneous nondegeneracy}. First, $F$ is positively homogeneous of degree $2\theta$, and $0<2\theta<1$. Thus, Assumption~\ref{ass: homogeneous nondegeneracy}~\ref{item ass_positive_homo} holds. Next, we prove that $DF$ is injective on the admissible class. Since $A(M)=(1-\tau)(\tr M)^2+4\tau\det M$ and, in dimension two, $\operatorname{cof}M=(\tr M)I-M$ for symmetric matrices, we have
		\begin{equation}
			DA(M)=2(1-\tau)(\tr M)I+4\tau\operatorname{cof}M=2(1+\tau)(\tr M)I-4\tau M.
		\end{equation}
		The map $M\mapsto DA(M)$ is a linear isomorphism on $\cS$. Indeed, writing $M=aI+M_0$ with $\tr M_0=0$, one obtains
		\[
		DA(M)=4aI-4\tau M_0,
		\]
		so the scalar and trace free parts are both recovered uniquely from $DA(M)$.
		
		Suppose that $M_1,M_2\in\mathcal A_\Gamma$ and $DF(M_1)=DF(M_2)$. Since $DF(M)=\theta A(M)^{\theta-1}DA(M)$ and $A(M)>0$ on $\mathcal A_\Gamma$, we obtain $DA(M_1)=\rho\,DA(M_2)$ for some $\rho>0$. Since $M\mapsto DA(M)$ is a linear isomorphism, this implies $M_1=\rho M_2$. Using the homogeneity of $DF$, we get $DF(M_1)=DF(\rho M_2)=\rho^{2\theta-1}DF(M_2)$. Since $DF(M_1)=DF(M_2)$ and $2\theta-1\neq0$, it follows that $\rho=1$. Hence, $M_1=M_2$, and Assumption~\ref{ass: homogeneous nondegeneracy}~\ref{item ass_DF_inj} holds.
		
		It remains to verify the non-proportionality condition. Fix $M\in\mathcal A_\Gamma$. Since $M$ is symmetric, we may work in an orthonormal basis in which $M=\operatorname{diag}(\lambda_1,\lambda_2)$. Choose
		\begin{equation}
			H=\begin{pmatrix}
				0&1\\
				1&0
			\end{pmatrix}.
		\end{equation}
		Then $\tr H=0$, and since $\operatorname{cof}M$ is diagonal, we have $DA(M)[H]=0$. Consequently, $DF(M)[H]=\theta A(M)^{\theta-1}DA(M)[H]=0$. On the other hand, $A(M+sH)=A(M)-4\tau s^2$, so $D^2A(M)[H,H]=-8\tau$. Therefore,
		\begin{equation}
			D^2F(M)[H,H]=\theta A(M)^{\theta-1}D^2A(M)[H,H]=-8\theta\tau A(M)^{\theta-1}\neq 0.
		\end{equation}
		If $D^2F(M)$ were proportional to $DF(M)\otimes DF(M)$, then there would exist a scalar $\gamma\in\mathbb R$ such that $D^2F(M)[H,H]=\gamma(DF(M)[H])^2$. For the above choice of $H$, the right hand side is zero, while the left hand side is nonzero. This is a contradiction. Hence $D^2F(M)$ is not proportional to $DF(M)\otimes DF(M)$, and Assumption~\ref{ass: homogeneous nondegeneracy}~\ref{item ass_non-prop} holds.
	\end{example}

	Before proving the uniqueness result, we record the part of the argument that does not use the non-proportionality condition. It shows that the equality of the DN maps already determines the two possible sources up to an explicit scalar factor.
	
	\begin{proposition}[Determination up to a gauge]\label{prop: determination up to gauge}
		Let the assumptions of {\rm Theorem~\ref{thm: main}} hold, except possibly the non-proportionality condition in {\rm Assumption~\ref{ass: homogeneous nondegeneracy}~\ref{item ass_non-prop}}. Let $u_0^{(j)}\in C^\infty(\overline{\Omega})$, $j=1,2$, be the admissible solutions to
		\begin{equation}\label{eq: main background zero gauge}
			\begin{cases}
				F(D^2u_0^{(j)})=f_j &\text{ in }\,\Omega, \\
				u_0^{(j)}=0 &\text{ on }\,\partial\Omega.
			\end{cases}
		\end{equation}
		If \eqref{eq: BC for the source} and \eqref{DN map same} hold, then there exists a positive function $c_0=c_0(x)$ in $\Omega$ such that
		\begin{equation}\label{eq: source gauge relation}
			f_1=c_0^{\frac{m}{m-1}}f_2
			\quad \text{in }\,\Omega.
		\end{equation}
	\end{proposition}

	\bigskip
	
	\noindent \textbf{Ideas of the proof.}
	The equality of the nonlinear DN maps allows us to compare the first linearizations at suitable reference solutions. In two dimensions, the inverse problem for the resulting non-divergence form linear equations determines the linearized coefficients only up to a conformal factor. Thus the first linearization gives $DF(D^2u_0^{(1)})=c_0DF(D^2u_0^{(2)})$ for a positive function $c_0$. This is the source of the factor appearing in Proposition~\ref{prop: determination up to gauge}. The homogeneity of $F$ and the injectivity of $DF$ then convert this relation into the source relation \eqref{eq: source gauge relation}.
	
	The main point is to show that this remaining factor is trivial under the non-proportionality condition. For this purpose, we use the second linearization. It produces an identity involving the second derivative tensor $D^2F$ and Hessians of solutions to the first linearized equation. A local construction of solutions, together with Runge approximation, gives enough Hessians at each interior point to identify the possible defect in the second derivative tensor. This yields the rank-one defect identity used in the final step. The non-proportionality condition then excludes the possibility that the conformal factor is nontrivial, and hence gives $c_0=1$ and $f_1=f_2$.
	
	This mechanism is the main novelty of the paper. The first linearization alone leaves a natural conformal ambiguity, while the second linearization detects enough additional information to remove this ambiguity under an algebraic condition on $F$. Thus the argument separates the inverse problem into two parts: determining the source up to an explicit factor, and eliminating that factor by using the structure of the nonlinearity.
	
	\bigskip
	
	\noindent \textbf{Earlier literature.}
	We also recall that inverse problems for nonlinear partial differential equations have attracted increasing attention in recent years. A classical starting point is the linearization of the nonlinear Dirichlet-to-Neumann map, going back to the work of Isakov \cite{Isa93}, which reduces the nonlinear inverse problem to a linear one. Later developments showed that higher order linearizations can reveal genuinely nonlinear effects that are invisible at first order; see, for instance, \cite{KN02, Sun96, Sun10}. A major development in this direction was the higher order linearization method introduced in nonlinear geometric settings in \cite{KLU2018}, and subsequently developed for semilinear elliptic equations in \cite{FO20, LLLS2019nonlinear}. Since then, this method has led to a broad range of uniqueness and parameter recovery results for nonlinear equations.
	
	In the elliptic setting, higher order linearization techniques have been applied to several semilinear inverse problems, including problems with partial data and related nonlinear interactions; see \cite{LLLS2019partial, LLST2022inverse, KU20a, KU20b, FLL23}. Quasilinear elliptic inverse problems have also been studied in \cite{KKU2022partial, CFKKU21, LW24_quasi}. For inverse problems related to the minimal surface equation and related quasilinear models, we refer for example to \cite{ABN20, CLLT24, Nur24}. We also mention the inverse source problems for semilinear elliptic and parabolic equations in \cite{LL24_elliptic_source, KLL_reaction_source}.
	
	Among nonlinear inverse problems, the recovery of source terms is of particular interest. In the linear case, such problems generally suffer from a natural obstruction, and uniqueness fails in general. In contrast, nonlinear equations may break this obstruction. Recent progress in this direction includes uniqueness results for inverse source problems for semilinear elliptic and parabolic equations \cite{LL24_elliptic_source, KLL_reaction_source}, and more recently for the fully nonlinear Monge-Amp\`ere equation in two dimensions \cite{LL2025IP_Monge_Ampere}. The present work continues this line of study for a broader class of admissible fully nonlinear elliptic equations.
	
	\bigskip
	
	\noindent \textbf{Organization of the article.}
	The remainder of this article is organized as follows. In Section~\ref{sec: prel}, we establish the well-posedness of the admissible Dirichlet problem and develop the higher-order linearization scheme. Section~\ref{sec:first linearization} is devoted to the analysis of the first linearization, including the boundary determination and the recovery of the linearized coefficients up to a conformal factor. In Section~\ref{sec:second linearization}, we study the second linearization and derive the key tensorial identities needed for the inverse problem. Finally, in Section~\ref{sec:inverse problem}, we combine these ingredients to prove Theorem~\ref{thm: main}.

	\section{Preliminaries}\label{sec: prel}
	
	\subsection{The well-posedness}\label{subsec: well-posedness}
	
	Let us demonstrate the well-posedness of \eqref{eq: main}. Given $\phi\in C^{\infty}(\partial\Omega)$, we introduce the following set of boundary data:
	\begin{equation}\label{admissible boundary data}
		B_{\phi,\delta}(\partial \Omega) := \left\{ \varphi \in C^{\infty}(\partial\Omega) : \| \varphi - \phi\|_{C^{4,\alpha}(\partial \Omega)} < \delta \right\},
	\end{equation}
	for some $\alpha \in (0,1)$ and sufficiently small $\delta > 0$. To solve the inverse problem \ref{Q:IP}, we employ the higher-order linearization technique, which requires the following well-posedness result describing the smooth dependence on the boundary data in the admissible class.
	
	\begin{proposition}\label{prop: well-posed}
		Let $\Omega \subset \R^2$ be a bounded open set with $C^\infty$-smooth boundary $\partial \Omega$. Let $F(M)=\widetilde f(\lambda(M))$, where $\widetilde f$ is a smooth symmetric function defined on an open convex cone $\Gamma\subset \R^2$ satisfying {\rm Assumption~\ref{assump1}}, and suppose that $\Omega$ satisfies \eqref{eq: CNS domain intro}. Let $f\in C^\infty(\overline{\Omega})$ be positive and satisfy \eqref{mini}. Given $\varphi\in C^\infty(\partial\Omega)$, the following statements hold:
		\begin{enumerate}[{\rm(i)}]
			\item\label{item 1 well-posed}
			There exists a unique admissible solution $u=u_\varphi\in C^\infty(\overline{\Omega})$ to the Dirichlet problem \eqref{eq: main}.
			
			\item\label{item 2 well-posed}
			Given $\phi \in C^\infty(\partial\Omega)$, there exist $\delta, C>0$ such that, for any $\varphi \in B_{\phi, \delta}(\partial \Omega)$ defined by \eqref{admissible boundary data}, there exists a unique admissible solution $u_\varphi\in C^{\infty}(\overline{\Omega})$ of \eqref{eq: main} satisfying
			\begin{equation}
				\|u_\varphi-u_\phi\|_{C^{4,\alpha}(\overline{\Omega})}
				\leq
				C\|\varphi-\phi\|_{C^{4,\alpha}(\partial \Omega)},
			\end{equation}
			where $u_\phi \in C^\infty(\overline{\Omega})$ is the solution to
			\begin{equation}\label{eq: main with phi Dirichlet data}
				\begin{cases}
					F(D^2u_\phi)=f &\text{ in }\Omega, \\
					u_\phi=\phi &\text{ on }\partial \Omega.
				\end{cases}
			\end{equation}
			
			\item\label{item 3 well-posed}
			Moreover, for $\varphi \in B_{\phi, \delta}(\partial \Omega)$, there exist $C^\infty$-Fr\'echet differentiable maps
			\begin{equation}\label{solution map}
				\begin{split}
					S&: B_{\phi,\delta}(\partial \Omega) \to C^{\infty}(\overline{\Omega}), \quad \varphi \mapsto u_{\varphi}, \\
					\Lambda&: B_{\phi,\delta}(\partial \Omega) \to C^{\infty}(\partial\Omega), \quad \varphi\mapsto \left. \partial_{\nu}u_{\varphi} \right|_{\partial\Omega}.
				\end{split}
			\end{equation}
		\end{enumerate}
	\end{proposition}
	
	For the well-posedness, we only use Assumption~\ref{assump1}, not the additional assumptions in Assumption~\ref{ass: homogeneous nondegeneracy} needed for the inverse problem.
	
	\begin{proof}[Proof of Proposition \ref{prop: well-posed}]
		For \ref{item 1 well-posed}, the existence and uniqueness of an admissible solution follow from \cite[Theorem~2]{CNS_nonlinear_Hessian}. The assumptions imposed above correspond to the hypotheses required there: the cone condition, symmetry, ellipticity, concavity, the boundary-side condition \eqref{mini}, the growth conditions, and the boundary admissibility condition \eqref{eq: CNS domain intro}. Therefore, \eqref{eq: main} admits a unique admissible classical solution. Since $f$, $\varphi$, and $\partial\Omega$ are smooth, standard elliptic regularity and bootstrap imply that this admissible solution is smooth up to the boundary. Hence $u_\varphi\in C^\infty(\overline{\Omega})$.
		
		For \ref{item 2 well-posed}, we first note that the solution $u_\phi \in C^\infty(\overline{\Omega})$ is guaranteed by \ref{item 1 well-posed}. We now study the local dependence of the admissible solution on the boundary data by the implicit function theorem for Banach spaces; see, for example, \cite[Theorem 10.6]{renardy2006introduction}. Let
		\[
		X:= C^{4,\alpha}(\overline{\Omega}), \quad Y:=C^{2,\alpha}(\overline{\Omega})\times C^{4,\alpha}(\partial \Omega).
		\]
		Since $u_\phi$ is admissible and the cone $\Gamma$ is open, there exists an open neighborhood $\mathcal U\subset X$ of $u_\phi$ such that $\lambda(D^2u(x))\in \Gamma$ for all $u\in \mathcal U$ and all $x\in \overline{\Omega}$. Thus every $u\in \mathcal U$ remains admissible. Consider the map
		\begin{equation}\label{mapping prop in well-posed}
			\Psi: C^{4,\alpha}(\partial\Omega)\times \mathcal U\to Y, \quad \Psi(\varphi,u):= \left(F(D^2 u), u|_{\partial\Omega}-\varphi \right).
		\end{equation}
		Since $F$ is smooth on the admissible class, the map $\Psi$ is well defined and $C^\infty$-smooth in the Fr\'echet sense.
		
		From \eqref{eq: main with phi Dirichlet data}, we have
		\[
		\Psi(\phi,u_\phi)=(F(D^2 u_\phi),u_\phi|_{\partial\Omega}-\phi)=(f,0),
		\]
		and the partial differential with respect to $u$ is given by
		\begin{equation}
			\begin{split}
				\partial_u \Psi(\phi,u_\phi)&: X\to Y, \\
				\partial_u \Psi (\phi,u_\phi)v  &= \big( F_{ab}(D^2u_\phi) \partial_{ab}v ,  v|_{\partial \Omega} \big),
			\end{split}
		\end{equation}
		for any $v\in X$. Since $u_\phi$ is admissible and smooth up to the boundary, the image $\{D^2u_\phi(x):x\in\overline{\Omega}\}$ is a compact subset of the admissible class. By the ellipticity condition in Assumption~\ref{assump1}{\rm(i)}, the matrix $\big(F_{ab}(D^2u_\phi(x))\big)$ is positive-definite for every $x\in\overline{\Omega}$, and hence uniformly positive-definite on $\overline{\Omega}$. Therefore $F_{ab}(D^2u_\phi)\partial_{ab}$ is a uniformly elliptic operator in $\Omega$.
		
		We now claim that
		\begin{equation}
			\partial_u \Psi (\phi,u_\phi): X \to Y, \quad v\mapsto \big( F_{ab}(D^2u_\phi) \partial_{ab}v ,  v|_{\partial \Omega} \big)
		\end{equation}
		is a linear isomorphism. Indeed, consider the Dirichlet problem
		\begin{equation}\label{equ Schauder in well-posed}
			\begin{cases}
				F_{ab}(D^2u_\phi)\partial_{ab}v= G &\text{ in }\Omega, \\
				v=\psi &\text{ on }\partial\Omega,
			\end{cases}
		\end{equation}
		where $G\in C^{2,\alpha}(\overline{\Omega})$ and $\psi\in C^{4,\alpha}(\partial\Omega)$. As in \cite[Section 2]{LL2025IP_Monge_Ampere}, and using the standard solvability theory for linear uniformly elliptic equations in non-divergence form together with the Schauder estimates in \cite[Chapter 6]{gilbarg2015elliptic}, the Dirichlet problem \eqref{equ Schauder in well-posed} has a unique solution $v\in C^{2,\alpha}(\overline{\Omega})$. Moreover, the global Schauder estimates imply that $v\in C^{4,\alpha}(\overline{\Omega})$. Hence \(\partial_u \Psi(\phi,u_\phi)\) is a linear isomorphism from $X$ onto $Y$.
		
		Therefore, by the implicit function theorem, there exist $\delta>0$ and a unique $C^\infty$-Fr\'echet differentiable map
		\[
		S: B_{\phi,\delta}(\partial \Omega) \to C^{\infty}(\overline{\Omega}), \quad \varphi \mapsto S(\varphi),
		\]
		such that $S(\phi)=u_\phi$ and $\Psi(\varphi, S(\varphi))=(f,0)$ for all $\varphi \in B_{\phi,\delta}(\partial \Omega)$, provided that $\delta>0$ is sufficiently small. Writing $u_\varphi:=S(\varphi)$, we obtain a unique admissible solution near $u_\phi$ for all nearby boundary data. Since $S$ is locally Lipschitz and $S(\phi)=u_\phi$, there exists a constant $C>0$ such that
		\begin{equation}
			\|u_\varphi-u_\phi\|_{C^{4,\alpha}(\overline{\Omega})} \leq C \|\varphi-\phi\|_{C^{4,\alpha}(\partial\Omega)}.
		\end{equation}
		This proves \ref{item 2 well-posed}.
		
		Finally, \ref{item 3 well-posed} follows from the $C^\infty$-Fr\'echet smoothness of the solution map $S$ and the fact that the normal trace map $u\mapsto \partial_\nu u|_{\partial\Omega}$ is continuous and linear from $C^{4,\alpha}(\overline{\Omega})$ to $C^{3,\alpha}(\partial\Omega)$. Hence it defines a smooth map into $C^\infty(\partial\Omega)$ on the smooth solutions under consideration. This completes the proof.
	\end{proof}

	\subsection{The higher order linearization}
	
	Inspired by \cite{LLLS2019nonlinear,FO20}, let us apply the higher order linearization technique to \eqref{eq: main}, with the Dirichlet data
	\begin{equation}\label{eq: Dirichlet data}
		\varphi=\phi+\varepsilon_1\phi_1+\varepsilon_2\phi_2 \text{ on } \partial\Omega,
	\end{equation}
	where $\phi,\phi_1,\phi_2\in C^\infty(\partial\Omega)$ are given functions, and $\varepsilon_k\in \R$ are parameters with $|\varepsilon_k|$ sufficiently small for $k=1,2$. By Proposition~\ref{prop: well-posed}, for sufficiently small $\varepsilon=(\varepsilon_1,\varepsilon_2)$, there exists a unique admissible solution $u_\varepsilon\in C^\infty(\overline{\Omega})$ to
	\begin{equation}\label{eq: main with epsilon}
		\begin{cases}
			F(D^2 u_\varepsilon)=f & \text{ in }\Omega, \\
			u_\varepsilon=\phi+\varepsilon_1 \phi_1+\varepsilon_2\phi_2 & \text{ on }\partial\Omega.
		\end{cases}
	\end{equation}
	
	\subsubsection{The first linearization}
	
	Let $u_\phi$ be the admissible solution to
	\begin{equation}\label{eq: main u_phi}
		\begin{cases}
			F(D^2u_\phi)=f & \text{ in }\Omega, \\
			u_\phi=\phi & \text{ on }\partial\Omega.
		\end{cases}
	\end{equation}
	Differentiating \eqref{eq: main with epsilon} with respect to $\varepsilon_k$ at $\varepsilon=0$, we obtain the first linearized equation
	\begin{equation}\label{eq: first linearization}
		\begin{cases}
			F_{ab}(D^2u_\phi)\partial_{ab} v_k =0 & \text{ in }\Omega,\\
			v_k=\phi_k & \text{ on }\partial\Omega,
		\end{cases}
	\end{equation}
	where
	\[
	v_k=\left. \partial_{\varepsilon_k} \right|_{\varepsilon=0} u_\varepsilon,
	\quad k=1,2.
	\]
	Since $u_\phi$ is admissible and smooth up to the boundary, the set $\{D^2u_\phi(x): x\in \overline{\Omega}\}$ is a compact subset of the admissible class. By the ellipticity condition in Assumption~\ref{assump1} \ref{item 1 - elliptic}, the matrix $\big(F_{ab}(D^2u_\phi(x))\big)$ is positive definite for every $x\in\overline{\Omega}$, and hence uniformly positive definite on $\overline{\Omega}$. Therefore the operator $F_{ab}(D^2u_\phi)\partial_{ab}$ is uniformly elliptic in $\Omega$.

	\subsubsection{The second linearization}

	Using the Dirichlet data \eqref{eq: Dirichlet data} and applying the second derivative $\p_{\eps_1\eps_2}\big|_{\eps=0}$ on the equation \eqref{eq: main with epsilon}, then one has the following second linearized equation 
	\begin{equation}\label{eq: second linearization}
		\begin{cases}
			g^{ab}\p_{ab} w + g^{ab,k\ell} \p_{ab}v_1 \p_{k\ell}v_2 =0  &\text{ in }\Omega, \\
			w=0 &\text{ on }\p\Omega,
		\end{cases}
	\end{equation}
	where $v_k$ is the solution to \eqref{eq: first linearization} for $k=1,2$, and $w=\left. \p_{\eps_1\eps_2} \right|_{\eps=0}u_\eps $. Here, 
	\begin{equation}\label{eq: metric gabkl}
		g^{ab} : = F_{ab}(D^2u_\phi)\;\;\text{and}\;\; g^{ab,k\ell} : = F_{ab,k\ell}(D^2u_\phi)\;\; \text{for}\;\; x\in\overline{\Omega},
	\end{equation}
	for all $a,b,k,\ell=1,2$, where $u_\phi$ is the admissible solution to \eqref{eq: main u_phi}.
	In the rest of the paper, we will adopt these notations to simplify the analysis.

	\section{The first linearization}\label{sec:first linearization}
	
	We first extract useful information from the first derivative of $F$ at suitable admissible solutions.
	
	\subsection{Boundary determination}
	
	Let us first show the boundary determination of $F$ on $\partial \Omega$. To this end, given $\phi \in C^\infty(\partial\Omega)$, let us consider the following admissible Dirichlet problem
	\begin{equation}\label{eq: main j=1,2 zero}
		\begin{cases}
			F(D^2 u_\phi^{(j)})=f_j &\text{ in }\Omega, \\
			u^{(j)}_\phi=\phi &\text{ on }\partial\Omega,
		\end{cases}
	\end{equation}
	for $j=1,2$.
	
	\begin{lemma}[Boundary determination]\label{lem: boundary determination}
		Adopting the assumptions in {\rm Theorem \ref{thm: main}}, the relations \eqref{eq: BC for the source} and \eqref{DN map same} imply
		\begin{equation}\label{eq: boundary determination of u_0}
			F_{ab}(D^2u_\phi^{(1)})\big|_{\partial\Omega}=F_{ab}(D^2u_\phi^{(2)})\big|_{\partial\Omega},
		\end{equation}
		where $F_{ab}$ is given by \eqref{eq: F_first derivative} for $1\le a,b\le 2$, and $u_\phi^{(j)}\in C^\infty(\overline{\Omega})$ is the admissible solution to \eqref{eq: main j=1,2 zero}, for $j=1,2$.
	\end{lemma}
	
	\begin{proof}
		Recall that $\mathcal A_\Gamma$ denotes the admissible class. Fix a boundary point $x_0\in \partial\Omega$. Without loss of generality, we may assume that $x_0=(0,0)$, and the boundary near $x_0$ is parameterized by $x_2=\eta(x_1)$ for $x_1\in (-\varepsilon,\varepsilon)$, where $\eta\in C^\infty((-\varepsilon,\varepsilon))$ and $\eta(0)=\eta'(0)=0$.
		
		Since $u_\phi^{(1)}=u_\phi^{(2)}=\phi$ on $\partial\Omega$ and the DN maps agree, we have $\partial_\nu u_\phi^{(1)}=\partial_\nu u_\phi^{(2)}$ on $\partial\Omega$. Hence the tangential and normal derivatives of $u_\phi^{(1)}$ and $u_\phi^{(2)}$ coincide on $\partial\Omega$, and therefore $\nabla u_\phi^{(1)}=\nabla u_\phi^{(2)}$ on $\partial\Omega$. In particular,
		\[
		\partial_1u_\phi^{(1)}(x_1,\eta(x_1))=\partial_1u_\phi^{(2)}(x_1,\eta(x_1))
		\]
		and
		\[
		\partial_2u_\phi^{(1)}(x_1,\eta(x_1))=\partial_2u_\phi^{(2)}(x_1,\eta(x_1))
		\]
		for $x_1$ near $0$. Differentiating these identities with respect to $x_1$ and evaluating at $x_1=0$, we obtain
		\[
		\partial_{11}u_\phi^{(1)}(0)+\eta'(0)\partial_{12}u_\phi^{(1)}(0)
		=
		\partial_{11}u_\phi^{(2)}(0)+\eta'(0)\partial_{12}u_\phi^{(2)}(0)
		\]
		and
		\[
		\partial_{12}u_\phi^{(1)}(0)+\eta'(0)\partial_{22}u_\phi^{(1)}(0)
		=
		\partial_{12}u_\phi^{(2)}(0)+\eta'(0)\partial_{22}u_\phi^{(2)}(0).
		\]
		Since $\eta'(0)=0$, it follows that
		\begin{equation}\label{eq: boundary second derivatives 11 12}
			\partial_{11}u_\phi^{(1)}(0)=\partial_{11}u_\phi^{(2)}(0)
			\quad \text{and} \quad
			\partial_{12}u_\phi^{(1)}(0)=\partial_{12}u_\phi^{(2)}(0).
		\end{equation}
		It remains to show that $\partial_{22}u_\phi^{(1)}(0)=\partial_{22}u_\phi^{(2)}(0)$.
		
		Write $t_j:=\partial_{22}u_\phi^{(j)}(0)$ for $j=1,2$, and define
		\[
		M_0=
		\begin{pmatrix}
			m_0^{(11)} & m_0^{(12)} \\
			m_0^{(12)} & 0
		\end{pmatrix},
		\]
		where $m_0^{(11)}=\partial_{11}u_\phi^{(1)}(0)=\partial_{11}u_\phi^{(2)}(0)$ and $m_0^{(12)}=\partial_{12}u_\phi^{(1)}(0)=\partial_{12}u_\phi^{(2)}(0)$. Then $D^2u_\phi^{(j)}(0)=M_0+t_j\,e_2\otimes e_2$ for $j=1,2$, where $e_2=(0,1)$. Since $u_\phi^{(j)}$ is admissible, both matrices $D^2u_\phi^{(j)}(0)$ belong to $\mathcal A_\Gamma$. Consider
		\[
		I:=\{t\in \mathbb R:\ M_0+t\,e_2\otimes e_2\in \mathcal A_\Gamma\}.
		\]
		Since $\mathcal A_\Gamma$ is an open convex subset of $\cS$, the set $I$ is an open interval containing $t_1$ and $t_2$.
		
		Define $g(t):=F(M_0+t\,e_2\otimes e_2)$ for $t\in I$. Then $g$ is well defined on $I$, and by the chain rule,
		\[
		g'(t)=F_{22}(M_0+t\,e_2\otimes e_2), \quad t\in I.
		\]
		By ellipticity in the admissible class, the matrix $\big(F_{ab}(M)\big)$ is positive definite for every $M\in \mathcal A_\Gamma$. In particular, $g'(t)=F_{22}(M_0+t\,e_2\otimes e_2)>0$ for all $t\in I$. Hence $g$ is strictly increasing on $I$. On the other hand, by construction,
		\[
		g(t_j)=F(D^2u_\phi^{(j)}(0))=f_j(0), \quad j=1,2.
		\]
		Since $x_0=0\in \partial\Omega$ and $f_1=f_2$ on $\partial\Omega$, we have $g(t_1)=f_1(0)=f_2(0)=g(t_2)$. Because $g$ is strictly increasing on $I$, it follows that $t_1=t_2$. Therefore $\partial_{22}u_\phi^{(1)}(0)=\partial_{22}u_\phi^{(2)}(0)$.
		
		Together with \eqref{eq: boundary second derivatives 11 12}, this implies
		\begin{equation}\label{eq: uniqueness of the Hessian on the boundary}
			D^2u_\phi^{(1)}(0)=D^2u_\phi^{(2)}(0).
		\end{equation}
		Since $x_0\in \partial\Omega$ was arbitrary, we conclude that $D^2u_\phi^{(1)}\big|_{\partial\Omega}=D^2u_\phi^{(2)}\big|_{\partial\Omega}$. Because $F$ is known a priori, this yields
		\[
		F_{ab}(D^2u_\phi^{(1)})\big|_{\partial\Omega}
		=
		F_{ab}(D^2u_\phi^{(2)})\big|_{\partial\Omega},
		\]
		for $1\le a,b\le 2$. This proves the lemma.
	\end{proof}
	
	\begin{remark}
		Note that \eqref{eq: uniqueness of the Hessian on the boundary} already implies not only
		\[
		F_{ab}(D^2u_\phi^{(1)})\big|_{\partial\Omega}=F_{ab}(D^2u_\phi^{(2)})\big|_{\partial\Omega},
		\]
		but also
		\[
		D^kF_{ab}(D^2u_\phi^{(1)})\big|_{\partial\Omega}=D^kF_{ab}(D^2u_\phi^{(2)})\big|_{\partial\Omega}
		\]
		for all $k\in \mathbb N$, since the nonlinear functional $F$ is known a priori. In addition, the concavity of $F$ is not used in the derivation of the boundary determination.
	\end{remark}
	
	\subsection{Interior recovery up to a conformal factor}
	
	First, since we know the DN maps of \eqref{eq: main}, we obtain the boundary information of $\nabla u|_{\partial\Omega}$ for any admissible solution $u$ to \eqref{eq: main}. Second, thanks to the boundary determination, we also know the matrix $\big(F_{ab}(D^2u_\phi(x))\big)_{1\le a,b\le 2}$ for $x\in \partial\Omega$. Thus, we can obtain the following alternative DN map for the linearized equation
	\begin{equation}\label{eq: DN first linearization}
		\Lambda'_g : C^\infty(\partial\Omega)\to C^\infty(\partial\Omega), \quad \phi \mapsto g^{ab}\partial_b v \nu_a \big|_{\partial\Omega},
	\end{equation}
	where $v\in C^\infty(\overline{\Omega})$ is the solution to \eqref{eq: first linearization}, $\nu=(\nu_1,\nu_2)$ denotes the unit outer normal on $\partial \Omega$.
	
	Now, we reduce our original inverse problem \ref{Q:IP} to the inverse problem for the first linearized equation \eqref{eq: first linearization}. More concretely, we use the information in \eqref{DN map same} to determine the metric $g$ whenever $\Omega\subset \R^2$ is a bounded simply connected domain. Therefore, we can apply \cite[Theorem 1.4]{LL2025IP_Monge_Ampere} to recover the metric up to a conformal factor $c>0$ in $\Omega$, with $c|_{\partial\Omega}=1$. For the reader's convenience, we record the following result, which gives the recovery of the metric $g$ in the non-divergence elliptic equation.
	
	\begin{proposition}[\text{\cite[Theorem 1.4]{LL2025IP_Monge_Ampere}}]\label{prop: J=Id}
		Let $\Omega \subset \R^2$ be a bounded open simply connected domain with $C^\infty$-smooth boundary $\partial\Omega$, and let $g=\big(g_{ab}\big)$ be a symmetric, positive definite, and $C^\infty$-smooth $2\times 2$ matrix-valued function and $(g^{ab})=(g_{ab})^{-1}$. Let $g=g_j$ be two metrics, and let
		\begin{equation}\label{eq: linearized DN map}
			\Lambda_{g_j}' : C^\infty(\partial\Omega)\to C^\infty(\partial\Omega), \quad h\mapsto g_j^{ab}\partial_b v^{(j)}\nu_a \big|_{\partial\Omega},
		\end{equation}
		be the DN map of the equation
		\begin{equation}\label{eq: non-divergence elliptic}
			\begin{cases}
				g^{ab}_j\partial_{ab} v^{(j)}=0 &\text{ in }\Omega,\\
				v^{(j)} = h &\text{ on }\partial \Omega,
			\end{cases}
		\end{equation}
		for $j=1,2$. Suppose that
		\begin{equation}
			\Lambda'_{g_1}(h)=\Lambda'_{g_2}(h) \quad \text{for all } h \in C^\infty(\partial \Omega),
		\end{equation}
		then there exists a $C^\infty$-smooth conformal factor $c=c(x)>0$ with $c|_{\partial \Omega}=1$ such that
		\begin{equation}\label{g_1 = c g_2 in intro}
			g_1^{ab}=cg_2^{ab} \text{ in }\Omega, \quad 1\le a,b\le 2.
		\end{equation}
	\end{proposition}
	
	The above proposition applies since the coefficient matrix $(g^{ab})$ is smooth and positive definite on $\overline{\Omega}$, hence uniformly elliptic on $\overline{\Omega}$.
	
	\begin{lemma}\label{lem: unique sol first linearize}
		Assume the hypotheses of Theorem~\ref{thm: main}, except possibly the non-proportionality condition in Assumption~\ref{ass: homogeneous nondegeneracy}~\ref{item ass_non-prop}. Let $v_k^{(j)}\in C^\infty(\overline{\Omega})$ be the solution to the first linearized equation
		\begin{equation}\label{eq: first linearization j=1,2}
			\begin{cases}
				F_{ab}(D^2u_\phi^{(j)})\partial_{ab} v^{(j)}_k=0 &\text{ in }\Omega,\\
				v^{(j)}_k=\phi_k &\text{ on }\partial \Omega,
			\end{cases}
		\end{equation}
		for $k=1,2$, where $u_\phi^{(j)}\in C^\infty(\overline{\Omega})$ is the admissible solution to
		\begin{equation}\label{eq: main u_0 j=1,2}
			\begin{cases}
				F(D^2u_\phi^{(j)})=f_j &\text{ in }\Omega, \\
				u_\phi^{(j)}=\phi &\text{ on }\partial \Omega.
			\end{cases}
		\end{equation}
		for $j=1,2$. Then condition \eqref{DN map same} implies $v_k^{(1)}=v_k^{(2)}$ in $\Omega$, for $k=1,2$.
	\end{lemma}
	
	\begin{proof}
		First, thanks to the boundary determination in Lemma \ref{lem: boundary determination}, we have $g_1^{ab}=g_2^{ab}$ on $\partial\Omega$, where
		\begin{equation}\label{eq: definition of g_jab}
			g_j^{ab}:= F_{ab}(D^2u^{(j)}_\phi), \quad a,b=1,2.
		\end{equation}
		Second, since we know \eqref{DN map same}, we obtain the boundary information
		\[
		\nabla u^{(1)}_\phi\big|_{\partial\Omega}=\nabla u^{(2)}_\phi\big|_{\partial\Omega},
		\]
		where $u_\phi^{(j)}\in C^\infty(\overline{\Omega})$ is the admissible solution to \eqref{eq: main u_0 j=1,2}, for $j=1,2$. In view of \eqref{eq: definition of g_jab}, it is clear that $g_j^{ab}$ depends on $\phi$, for $a,b,j=1,2$.
		
		Since we know $g_1^{ab}=g_2^{ab}$ on $\partial \Omega$, for $a,b=1,2$, and the Dirichlet data $\phi_k$ can be chosen arbitrarily in $C^\infty(\partial\Omega)$, the first linearization of \eqref{DN map same} gives
		\[
		\Lambda_{g_1}'(h)=\Lambda_{g_2}'(h)	\quad \text{for all } h\in C^\infty(\partial\Omega).
		\]
		Applying Proposition \ref{prop: J=Id}, we can recover the leading coefficients for the first linearized equation \eqref{eq: first linearization j=1,2} up to a conformal factor. More precisely, there exists a conformal factor $c_\phi=c_\phi(x)>0$ such that $c_\phi|_{\partial\Omega}=1$ and
		\begin{equation}\label{eq: gauge invariant up to c_phi}
			F_{ab}(D^2 u_\phi^{(1)})=c_\phi F_{ab}(D^2 u_\phi^{(2)}) \text{ in }\Omega.
		\end{equation}
		
		On the one hand, for each $k=1,2$, we have
		\begin{equation}
			\begin{cases}
				g_1^{ab} \partial_{ab}v^{(1)}_k = c_\phi g_2^{ab}\partial_{ab} v_{k}^{(1)}=0 &\text{ in }\Omega,\\
				v_k^{(1)}=\phi_k &\text{ on }\partial\Omega,
			\end{cases}
		\end{equation}
		and hence we may rewrite the equation as
		\begin{equation}\label{eq: first v_1}
			\begin{cases}
				g_2^{ab} \partial_{ab}v_{k}^{(1)}=0 &\text{ in }\Omega,\\
				v_k^{(1)}=\phi_k &\text{ on }\partial\Omega,
			\end{cases}
		\end{equation}
		since $c_\phi>0$ in $\overline{\Omega}$. On the other hand, $v_k^{(2)}$ also solves the Dirichlet problem
		\begin{equation}\label{eq: first v_2}
			\begin{cases}
				g_2^{ab}\partial_{ab} v_{k}^{(2)}=0 &\text{ in }\Omega,\\
				v_k^{(2)}=\phi_k &\text{ on }\partial\Omega,
			\end{cases}
		\end{equation}
		for $k=1,2$. Therefore, $v_k^{(1)}$ and $v_k^{(2)}$ solve the same equation with the same Dirichlet data, which implies that $v_k^{(1)}=v_k^{(2)}$ in $\overline{\Omega}$, for $k=1,2$. This proves the assertion.
	\end{proof}
	
	For the forthcoming analysis, let us adopt the notation
	\begin{equation}\label{eq: definition of v_k}
		v_k := v_k^{(1)}=v_k^{(2)} \text{ in }\overline{\Omega}, \text{ for }k=1,2,
	\end{equation}
	which will always denote the solution to the first linearized equation \eqref{eq: first linearization j=1,2}.
	
	\subsection{Determination up to a gauge}
	
	Thanks to the first linearization, we can prove Proposition \ref{prop: determination up to gauge}.
	
	\begin{proof}[Proof of {\rm Proposition~\ref{prop: determination up to gauge}}]
		Let $u_0^{(j)}\in C^\infty(\overline{\Omega})$, $j=1,2$, be the admissible solutions to \eqref{eq: main background zero gauge}. By Proposition~\ref{prop: J=Id} and Lemma~\ref{lem: unique sol first linearize}, there exists a positive function $c_0=c_0(x)$ such that
		\begin{equation}\label{eq: gauge proof first derivative}
			F_{ab}(D^2u_0^{(1)})=c_0 F_{ab}(D^2u_0^{(2)})
			\quad \text{in }\; \Omega.
		\end{equation}
		Fix $x\in\Omega$ and set
		\[
		M_1:=D^2u_0^{(1)}(x),
		\quad
		M_2:=D^2u_0^{(2)}(x).
		\]
		Since $u_0^{(1)}$ and $u_0^{(2)}$ are admissible, both $M_1$ and $M_2$ belong to the admissible class. From \eqref{eq: gauge proof first derivative}, we have
		\[
		DF(M_1)=c_0(x)DF(M_2).
		\]
		Since $F$ is positively homogeneous of degree $m\neq 1$, its first derivative is homogeneous of degree $m-1$. Thus
		\[
		DF(tM_2)=t^{m-1}DF(M_2)
		\]
		for every $t>0$. Define $t(x):=c_0(x)^{1/(m-1)}$. Then
		\[
		DF(M_1)=DF(t(x)M_2).
		\]
		Since the admissible class is a cone, $t(x)M_2$ is admissible. By Assumption~\ref{ass: homogeneous nondegeneracy}~\ref{item ass_DF_inj}, we obtain
		\[
		M_1=t(x)M_2.
		\]
		Using the homogeneity of $F$ again, we get
		\[
		f_1(x)=F(M_1)=F(t(x)M_2)=t(x)^mF(M_2)=c_0(x)^{\frac{m}{m-1}}f_2(x).
		\]
		Since $x\in\Omega$ was arbitrary, this proves \eqref{eq: source gauge relation}.
	\end{proof}

	\section{The second linearization}\label{sec:second linearization}
	
	For the second linearization, let $\zeta \in C^\infty(\partial\Omega)$ be given, and consider the one-parameter family of boundary data
	\[
	\phi=\delta \zeta \text{ on }\p\Omega
	\]
	in \eqref{eq: Dirichlet data}, where $\delta\in (-\varepsilon,\varepsilon)$ and $\varepsilon>0$ is sufficiently small. For $j=1,2$, let $u_\delta^{(j)}$ denote the admissible solution to
	\begin{equation}
		\begin{cases}
			F(D^2u_{\delta}^{(j)})= f_j &\text{ in }\Omega, \\
			u_{\delta}^{(j)}=\delta \zeta &\text{ on }\partial\Omega.
		\end{cases}
	\end{equation}
	Then the identity \eqref{eq: gauge invariant up to c_phi} can be rewritten as
	\begin{equation}\label{eq: gauge invariant up to c_delta}
		F_{ab}(D^2 u_{\delta}^{(1)})=c_\delta\, F_{ab}(D^2 u_{\delta}^{(2)})
		\quad \text{in}\;\,\Omega.
	\end{equation}
	Furthermore, since the solution maps depend smoothly on $\delta$ and the conformal factor in \eqref{eq: gauge invariant up to c_delta} is uniquely determined, $c_\delta$ depends smoothly on $\delta$.
	We now differentiate \eqref{eq: gauge invariant up to c_delta} with respect to $\delta$ at $\delta=0$. This yields
	\begin{equation}\label{eq: gauge invariant derivative 1}
		\begin{split}
			F_{ab,k\ell}(D^2u_0^{(1)})\, \partial_{k\ell} v^{(1)}
			=
			c_0\, F_{ab,k\ell}(D^2u_0^{(2)})\, \partial_{k\ell} v^{(2)}
			+\dot c\, F_{ab}(D^2u_0^{(2)})
			\quad \text{in }\Omega,
		\end{split}
	\end{equation}
	where
	\[
	c_0:=c_\delta|_{\delta=0},
	\quad
	\dot c:=\partial_\delta c_\delta|_{\delta=0}.
	\]
	Here $u_0^{(j)}$ is the admissible solution of
	\begin{equation}\label{eq: main backgroud =0}
		\begin{cases}
			F(D^2u_0^{(j)})=f_j &\text{ in }\Omega, \\
			u_0^{(j)}=0 &\text{ on }\partial\Omega,
		\end{cases}
	\end{equation}
	and $v^{(j)}$ denotes the solution to the first linearized equation
	\begin{equation}\label{eq: first linearized equation delta}
		\begin{cases}
			F_{ab}(D^2u_0^{(j)}) \partial_{ab} v^{(j)}=0 &\text{ in }\Omega, \\
			v^{(j)}=\zeta &\text{ on }\partial\Omega,
		\end{cases}
	\end{equation}
	for $j=1,2$.
	
	By Lemma~\ref{lem: unique sol first linearize}, there exists a conformal factor $\tilde c>0$ in $\Omega$ such that
	\[
	F_{ab}(D^2u_0^{(1)})=\tilde c F_{ab}(D^2u_0^{(2)})
	\quad \text{in }\Omega.
	\]
	Since the two equations in \eqref{eq: first linearized equation delta} have the same boundary value $\zeta$, uniqueness for second order elliptic equations implies
	\[
	v:=v^{(1)}=v^{(2)}
	\quad \text{in }\,\Omega.
	\]
	Substituting this into \eqref{eq: gauge invariant derivative 1}, we obtain
	\begin{equation}\label{eq: gauge invariant derivative 2}
		\big(
		F_{ab,k\ell}(D^2u_0^{(1)})
		-
		c_0F_{ab,k\ell}(D^2u_0^{(2)})
		\big)\partial_{k\ell}v
		=
		\dot cF_{ab}(D^2u_0^{(2)})
		\quad \text{in }\,\Omega,
	\end{equation}
	for $a,b=1,2$.
	\begin{lemma}\label{lem: c independent of boundary}
		For every choice of $\zeta\in C^\infty(\partial\Omega)$, one has $\dot c=0$.
	\end{lemma}
	
	\begin{proof}
		Fix $\zeta\in C^\infty(\partial\Omega)$. For each sufficiently small $\delta$, the conformal factor $c_\delta$ satisfies
		\[
		DF(D^2u_\delta^{(1)})=c_\delta DF(D^2u_\delta^{(2)})
		\quad \text{in }\, \Omega.
		\]
		Since $F$ is positively homogeneous of degree $m\neq 1$, its first derivative is homogeneous of degree $m-1$. Setting $t_\delta:=c_\delta^{1/(m-1)}$, we have
		\[
		DF(D^2u_\delta^{(1)})=DF(t_\delta D^2u_\delta^{(2)}).
		\]
		By the injectivity of $DF$ on the admissible class,
		\[
		D^2u_\delta^{(1)}=t_\delta D^2u_\delta^{(2)}
		\quad \text{in }\, \Omega.
		\]
		Using the homogeneity of $F$, we obtain
		\[
		f_1=F(D^2u_\delta^{(1)})=t_\delta^mF(D^2u_\delta^{(2)})
		=
		c_\delta^{\frac{m}{m-1}}f_2
		\quad \text{in }\, \Omega.
		\]
		Since $f_1$ and $f_2$ are independent of $\delta$, the last identity shows that $c_\delta$ is independent of $\delta$, and hence, $\dot c=0$ in $\Omega$.
	\end{proof}
	
	\subsection{Spanning Hessians and Runge approximation}	
	
	Consequently, \eqref{eq: gauge invariant derivative 2} simplifies to
	\begin{equation}\label{eq: gauge invariant derivative 4}
		\big(
		F_{ab,k\ell}(D^2u_0^{(1)})
		-
		c_0\,F_{ab,k\ell}(D^2u_0^{(2)})
		\big)\partial_{k\ell}v
		=0
		\quad \text{in }\,\Omega,
	\end{equation}
	for every solution $v$ to \eqref{eq: first linearized equation delta} and every $a,b=1,2$.
	
	To proceed further from \eqref{eq: gauge invariant derivative 4}, one needs to understand how much information about the tensor
	\[
	T^{ab}_{k\ell}:=
	F_{ab,k\ell}(D^2u_0^{(1)})
	-
	c_0F_{ab,k\ell}(D^2u_0^{(2)})
	\]
	can be extracted by testing against Hessians of solutions to \eqref{eq: first linearized equation delta}. At each point $x\in \Omega$, let
	\[
	A(x):=\big(A_{k\ell}(x)\big)_{1\le k,\ell\le 2}
	:=\big(F_{k\ell}(D^2u_0^{(2)}(x))\big)_{1\le k,\ell\le 2}.
	\]
	Since every solution $v$ to \eqref{eq: first linearized equation delta} satisfies
	\[
	A_{k\ell}(x)\partial_{k\ell}v(x)=0,
	\]
	the Hessian $D^2v(x)$ necessarily belongs to the subspace
	\[
	K_{A(x)}:=
	\big\{H=(H_{k\ell})\in \cS:\,A_{k\ell}(x)H_{k\ell}=0\big\}.
	\]
	The following elementary linear algebra lemma shows that, once sufficiently many such Hessians are available, the tensor $T^{ab}_{k\ell}(x)$ must be proportional to $A_{k\ell}(x)$ in the $(k,\ell)$-indices.
	
	\begin{lemma}\label{lem: proportionality from orthogonality}
		Let $A=(A_{k\ell}(x))$ be a non-degenerate matrix-valued function in $\Omega$ with $A(x)\in \cS$ for every $x\in \Omega$. For each $x\in \Omega$, define $K_{A(x)}$ as above.	 Suppose that for each fixed $1\le a,b\le 2$, the symmetric matrix-valued function
		\[
		T^{ab}(x):=\big(T^{ab}_{k\ell}(x)\big)_{1\le k,\ell\le 2}
		\]
		satisfies
		\[
		T^{ab}_{k\ell}(x)H_{k\ell}=0
		\quad \text{for all } H\in K_{A(x)},
		\]
		for every $x\in \Omega$. Then there exists a scalar function $\lambda^{ab}=\lambda^{ab}(x)$ such that
		\[
		T^{ab}_{k\ell}(x)=\lambda^{ab}(x)A_{k\ell}(x)
		\quad \text{for all } x\in \Omega,\ \ 1\le k,\ell\le 2.
		\]
	\end{lemma}
	
	\begin{proof}
		Fix $x\in \Omega$. Since $\cS$ is a three-dimensional real inner product space with respect to the Frobenius product, i.e.,
		\[
		M:N:=M_{k\ell}N_{k\ell},
		\]
		and $A(x)\neq 0$, the set $K_{A(x)}$ is a codimension-one linear subspace of $\cS$. Hence, its orthogonal complement $K_{A(x)}^\perp$ is one-dimensional. By definition of $K_{A(x)}$, one has
		\[
		A(x):H=0
		\quad \text{for all } H\in K_{A(x)},
		\]
		so that $A(x)\in K_{A(x)}^\perp$. Therefore
		\[
		K_{A(x)}^\perp=\mathrm{span}\{A(x)\}.
		\]
		Now fix $1\le a,b\le 2$. The assumption implies that
		\[
		T^{ab}(x):H=0
		\quad \text{for all } H\in K_{A(x)},
		\]
		namely $T^{ab}(x)\in K_{A(x)}^\perp$. Since $K_{A(x)}^\perp=\mathrm{span}\{A(x)\}$, there exists a scalar $\lambda^{ab}(x)$ such that
		\[
		T^{ab}(x)=\lambda^{ab}(x)A(x).
		\]
		In components, this reads
		\[
		T^{ab}_{k\ell}(x)=\lambda^{ab}(x)A_{k\ell}(x),
		\]
		as claimed.
	\end{proof}
	
	The next few results are stated for a general uniformly elliptic second-order operator with sufficiently regular coefficients. They will later be applied to the first linearized operator $F_{k\ell}(D^2u_0^{(2)})\partial_{k\ell}$,
	but their validity does not rely on this particular form.
	
	\begin{proposition}\label{prop: spanning Hessians at a point}
		Let
		\[
		Lv:=A_{k\ell}(x)\partial_{k\ell}v
		\quad \text{in } \Omega,
		\]
		where $A=(A_{k\ell})_{1\le k,\ell\le 2}\in C^\alpha(\Omega;\cS)$ is uniformly elliptic for some $\alpha\in (0,1)$. Fix $x_0\in \Omega$, and define
		\[
		K_{A(x_0)}:=\big\{H\in \cS:\,A_{k\ell}(x_0)H_{k\ell}=0\big\}.
		\]
		Then there exist a radius $r>0$ with $\overline{B_r(x_0)}\subset \Omega$ and two functions $v_1,v_2\in C^{2,\alpha}(B_r(x_0))\cap C(\overline{B_r(x_0)})$ solving
		\[
		Lv_j=0
		\quad \text{in }\, B_r(x_0), \quad j=1,2,
		\]
		such that $D^2v_1(x_0)$ and $D^2v_2(x_0)$ are linearly independent elements of $K_{A(x_0)}$. In particular,
		\[
		\mathrm{span}\big\{D^2v(x_0):\,Lv=0 \text{ near }x_0\big\}=K_{A(x_0)}.
		\]
	\end{proposition}
	
	\begin{proof}
		By translation, we may assume that $x_0=0$. Since $A(0)$ is symmetric positive definite by uniform ellipticity, there exists an invertible matrix $P$ such that
		\begin{equation*}
			P^{-1}A(0)P^{-T}=I_2.
		\end{equation*}
		Equivalently, $A(0)=PP^T$. Introduce the linear change of variables $x=Pz$, and, for a function $v$ in the $x$-variables, define $\widetilde v(z):=v(Pz)$. Then
		\begin{equation*}
			D_z^2\widetilde v(z)=P^T D_x^2v(x)P,
			\quad x=Pz.
		\end{equation*}
		Equivalently,
		\begin{equation*}
			D_x^2v(x)=P^{-T}D_z^2\widetilde v(z)P^{-1}.
		\end{equation*}
		Since the change of variables is linear, no first order terms are produced. Hence the equation $A(x):D_x^2v=0$ is transformed into
		\begin{equation*}
			\widetilde A(z):D_z^2\widetilde v=0,
		\end{equation*}
		where
		\begin{equation*}
			\widetilde A(z):=\big(\widetilde A_{k\ell}(z)\big)_{1\le k,\ell\le 2}
			:=P^{-1}A(Pz)P^{-T}.
		\end{equation*}
		In particular,
		\begin{equation*}
			\widetilde A(0)=P^{-1}A(0)P^{-T}=I_2.
		\end{equation*}
		
		We next compare the corresponding constraint spaces. The map $H\mapsto P^THP$ is a linear isomorphism from $K_{A(0)}$ onto $K_{I_2}$. Indeed, if $H\in K_{A(0)}$, then
		\begin{equation*}
			I_2:(P^THP)=\tr(P^THP)=\tr(PP^TH)=A(0):H=0.
		\end{equation*}
		Conversely, if $\widetilde H\in K_{I_2}$, then $H=P^{-T}\widetilde H P^{-1}$ satisfies
		\begin{equation*}
			A(0):H
			=
			A(0):P^{-T}\widetilde H P^{-1}
			=
			I_2:\widetilde H
			=
			0.
		\end{equation*}
		Therefore, it suffices to prove the proposition in the transformed coordinates. Renaming the variable $z$ by $x$, we may assume without loss of generality that $A(0)=I_2$.
		
		In this case,
		\begin{equation*}
			K_{A(0)}=\{H\in \cS:\,\tr(H)=0\},
		\end{equation*}
		which is a two-dimensional subspace of $\cS$. Consider the trace-free matrices
		\begin{equation*}
			H_1=
			\begin{pmatrix}
				1&0\\
				0&-1
			\end{pmatrix},
			\quad
			H_2=
			\begin{pmatrix}
				0&1\\
				1&0
			\end{pmatrix},
		\end{equation*}
		which form a basis of $K_{A(0)}$. Let
		\begin{equation*}
			q_1(x):=\frac12(x_1^2-x_2^2),
			\quad
			q_2(x):=x_1x_2.
		\end{equation*}
		Then $D^2q_1=H_1$, $D^2q_2=H_2$, and since $\Delta q_j=0$, we have
		\begin{equation*}
			A(0):D^2 q_j= I_2:D^2q_j=0,
			\quad j=1,2.
		\end{equation*}
		
		Choose $r_0>0$ so that $\overline{B_{r_0}(0)}\subset \Omega$. For each $r\in (0,r_0)$ and each $j=1,2$, let $v_{j,r}$ be the unique solution to
		\begin{equation}\label{eq: local Dirichlet problem for vjr}
			\begin{cases}
				A_{ab}(x)\partial_{ab}v_{j,r}=0 & \text{in } B_r(0),\\
				v_{j,r}=q_j & \text{on } \partial B_r(0).
			\end{cases}
		\end{equation}
		By the standard solvability theory for uniformly elliptic non-divergence form equations with $C^\alpha$ coefficients,
		\begin{equation*}
			v_{j,r}\in C^{2,\alpha}(B_r(0))\cap C(\overline{B_r(0)}).
		\end{equation*}
		Define the rescaled functions
		\begin{equation*}
			w_{j,r}(y):=r^{-2}v_{j,r}(ry), \quad y\in B_1(0).
		\end{equation*}
		A direct computation shows that $w_{j,r}$ solves
		\begin{equation}\label{eq: rescaled equation}
			\begin{cases}
				A_{ab,r}(y)\partial_{ab}w_{j,r}=0 & \text{in } B_1(0),\\
				w_{j,r}(y)=r^{-2}q_j(ry)=q_j(y) & \text{on } \partial B_1(0),
			\end{cases}
		\end{equation}
		where
		\begin{equation*}
			A_{ab,r}(y):=A_{ab}(ry)
			\quad \text{and} \quad
			A_r(y)=(A_{ab,r}(y))_{1\le a,b\le 2}.
		\end{equation*}
		Since $A\in C^\alpha(\Omega)$ and $A(0)=I_2$, it follows that
		\begin{equation*}
			A_r\to I_2
			\quad \text{in }\, C^\alpha(\overline{B_1(0)})
			\quad \text{as } r\to 0.
		\end{equation*}
		
		Let $\eta_{j,r}:=w_{j,r}-q_j$. Then $\eta_{j,r}=0$ on $\partial B_1(0)$ and
		\begin{equation*}
			A_{ab,r}(y)\partial_{ab}\eta_{j,r}
			=
			-\big(A_{ab,r}(y)-\delta_{ab}\big)\partial_{ab}q_j
			\quad \text{in }\, B_1(0).
		\end{equation*}
		The right hand side tends to zero in $C^\alpha(\overline{B_1(0)})$ as $r\to 0$. By the maximum principle and the interior Schauder estimates, for every $\rho\in (0,1)$,
		\begin{equation*}
			\eta_{j,r}\to 0
			\quad \text{in } C^2(B_\rho(0))
			\quad \text{as } r\to 0.
		\end{equation*}
		In particular, taking $\rho=1/2$, we obtain
		\begin{equation*}
			w_{j,r}\to q_j
			\quad \text{in } C^2(B_{1/2}(0))
			\quad \text{as } r\to 0.
		\end{equation*}
		Therefore, as $r\to 0$,
		\begin{equation*}
			D^2w_{j,r}(0)\to D^2q_j=H_j,
			\quad j=1,2.
		\end{equation*}
		On the other hand, by the definition of $w_{j,r}$,
		\begin{equation*}
			D^2w_{j,r}(0)=D^2v_{j,r}(0).
		\end{equation*}
		Hence
		\begin{equation*}
			D^2v_{j,r}(0)\to H_j,
			\quad j=1,2.
		\end{equation*}
		
		Since $H_1$ and $H_2$ are linearly independent, it follows that for all sufficiently small $r>0$, the matrices $D^2v_{1,r}(0)$ and $D^2v_{2,r}(0)$ are linearly independent. Moreover, since each $v_{j,r}$ solves $Lv_{j,r}=0$ in $B_r(0)$, evaluating the equation at $0$ gives
		\begin{equation*}
			A_{k\ell}(0)\partial_{k\ell}v_{j,r}(0)=0,
		\end{equation*}
		so that
		\begin{equation*}
			D^2v_{j,r}(0)\in K_{A(0)},
			\quad j=1,2.
		\end{equation*}
		Thus, for some sufficiently small $r>0$, the two solutions $v_1:=v_{1,r}$ and $v_2:=v_{2,r}$ satisfy that $D^2v_1(0)$ and $D^2v_2(0)$ are linearly independent elements of $K_{A(0)}$. Since $K_{A(0)}$ is two-dimensional, this implies
		\begin{equation*}
			\mathrm{span}\{D^2v(0):\,Lv=0 \text{ near }0\}=K_{A(0)}.
		\end{equation*}
		
		Finally, returning to the original variables $x=Pz$, if $\widetilde v$ is a solution in the transformed variables, then $v(x):=\widetilde v(P^{-1}x)$ solves the original equation near $0$. Moreover,
		\begin{equation*}
			D_z^2\widetilde v(0)=P^T D_x^2v(0)P.
		\end{equation*}
		The map $H\mapsto P^THP$ is a linear isomorphism from $K_{A(0)}$ onto $K_{I_2}$, and hence linear independence of the Hessians in the normalized variables is equivalent to linear independence of the corresponding Hessians in the original variables. After restricting to a sufficiently small Euclidean ball contained in the image of the transformed neighborhood, we obtain the conclusion at $0$.
	\end{proof}
	
	The preceding discussion shows that, to extract more information from \eqref{eq: gauge invariant derivative 4}, it is enough to produce sufficiently many global solutions to the first linearized equation whose Hessians at a fixed point span the subspace
	\[
	K_{A(x_0)}
	:=\{H\in \cS:\,A_{k\ell}(x_0)H_{k\ell}=0\},
	\]
	where
	\[
	A(x):=\big(A_{k\ell}(x)\big)_{1\le k,\ell\le 2}
	:=\big(F_{k\ell}(D^2u_0^{(2)}(x))\big)_{1\le k,\ell\le 2}.
	\]
	
	By Proposition~\ref{prop: spanning Hessians at a point}, such a spanning property holds locally near any fixed point $x_0\in \Omega$. We next invoke a Runge approximation result to replace these local solutions by global ones.
	
	\begin{proposition}[Runge approximation]\label{prop: weak Runge approximation}
		Let $\Omega_0\Subset \Omega\subset \R^2$ be a smooth subdomain, and let
		\[
		L:=A_{k\ell}(x)\partial_{k\ell}
		\quad \text{in } \Omega,
		\]
		where $A=(A_{k\ell})_{1\le k,\ell\le 2}\in C^\infty(\overline{\Omega};\cS)$ is uniformly elliptic. Let $w\in H^{2}(\Omega_0)$ satisfy $Lw=0$ in $\Omega_0$. Then for every compact set $K\subset \Omega_0$ and every $\varepsilon>0$, there exists a global solution $v\in H^{2}(\Omega)$ of $Lv=0$ in $\Omega$, such that
		\[
		\|v-w\|_{L^2(K)}<\varepsilon.
		\]
	\end{proposition}
	
	\begin{proof}
		The result follows from the standard Hahn-Banach duality argument combined with unique continuation for the adjoint operator. Since $L=A_{k\ell}(x)\partial_{k\ell}$, its formal adjoint is given by
		\[
		L^\ast \varphi=\partial_{k\ell}(A_{k\ell}\varphi)
		=\partial_k(A_{k\ell}\partial_\ell\varphi)+\partial_k\big((\partial_\ell A_{k\ell})\varphi\big).
		\]
		
		Let
		\[
		\mathcal S(\Omega):=\{v\in H^{2}(\Omega): Lv=0 \text{ in } \Omega\},
		\quad
		\mathcal S(\Omega_0):=\{u\in H^{2}(\Omega_0): Lu=0 \text{ in } \Omega_0\}.
		\]
		We claim that
		\[
		\overline{\{v|_{\Omega_0}: v\in \mathcal S(\Omega)\}}^{\,L^2(\Omega_0)}
		=
		\mathcal S(\Omega_0).
		\]
		By the Hahn-Banach theorem, it suffices to show the following: if $f\in L^2(\Omega_0)$ satisfies
		\begin{equation}
			\int_{\Omega_0} f\,v\,dx=0
			\quad \text{for all } v\in \mathcal S(\Omega)
			\implies
			\int_{\Omega_0} f\,u\,dx=0
			\quad \text{for all } u\in \mathcal S(\Omega_0).
		\end{equation}
		Extend $f$ by zero outside $\Omega_0$, and denote the extension by $\widetilde f\in L^2(\Omega)$. Since the Dirichlet problem for $L$ is uniquely solvable, the corresponding Dirichlet realization of $L$ is invertible. Hence, by the Fredholm alternative for elliptic boundary value problems, the adjoint Dirichlet problem is also uniquely solvable. Therefore, there exists
		\[
		\varphi\in H^1_0(\Omega)\cap H^2(\Omega)
		\]
		such that
		\[
		L^\ast \varphi=\widetilde f
		\quad \text{in } \Omega.
		\]
		
		Now let $v\in \mathcal S(\Omega)$ be arbitrary, and let $g:=v|_{\partial\Omega}$. Using the Green identity for $L$ and $L^\ast$, together with the fact that $\varphi=0$ on $\partial\Omega$, we obtain
		\begin{equation}\label{eq: Green identity Runge}
			\int_{\Omega_0} f\,v\,dx
			=
			\int_\Omega \widetilde f\,v\,dx
			=
			\int_\Omega (L^\ast\varphi)\,v\,dx-\int_\Omega \varphi\,Lv\,dx
			=
			\int_{\partial\Omega} \partial_{\nu_A}\varphi \, g\,dS,
		\end{equation}
		where $\partial_{\nu_A}\varphi:=\nu_kA_{k\ell}\partial_\ell\varphi$ is the conormal derivative associated with $L^\ast$. Here we used that $\varphi=0$ on $\partial\Omega$. By assumption,
		\[
		\int_{\Omega_0} f\,v\,dx=0
		\quad \text{for all } v\in \mathcal S(\Omega).
		\]
		Since smooth boundary data $g$ are arbitrary and give rise to solutions $v\in \mathcal S(\Omega)$, \eqref{eq: Green identity Runge} implies $\partial_{\nu_A}\varphi=0$ on $\partial\Omega$. Therefore,
		\[
		\varphi=\partial_{\nu_A}\varphi=0 \quad \text{on } \partial\Omega.
		\]
		
		Next, extend $\varphi$ by zero to a slightly larger smooth domain $\widetilde\Omega\supset \Omega$. Since both $\varphi$ and its conormal derivative vanish on $\partial\Omega$, the zero extension still satisfies
		\[
		L^\ast\varphi=0
		\quad \text{in } \widetilde\Omega\setminus \overline{\Omega_0}
		\]
		in the weak sense. Hence, by the unique continuation principle for second-order elliptic equations, we conclude that
		\[
		\varphi=0
		\quad \text{in } \Omega\setminus \overline{\Omega_0}.
		\]
		In particular, the traces of $\varphi$ and $\nabla\varphi$ vanish on $\partial\Omega_0$.
		
		Now let $u\in \mathcal S(\Omega_0)$ be arbitrary. Using the Green identity again, now in $\Omega_0$, we obtain
		\[
		\int_{\Omega_0} f\,u\,dx
		=
		\int_{\Omega_0} (L^\ast\varphi)\,u\,dx-\int_{\Omega_0} \varphi\,Lu\,dx
		=
		\int_{\partial\Omega_0} \big( \partial_{\nu_A}\varphi\,u-\varphi\,\partial_{\nu_A}u \big)\,dS.
		\]
		Since $\varphi$ and $\nabla\varphi$ vanish on $\partial\Omega_0$, the boundary integral is zero. Hence
		\[
		\int_{\Omega_0} f\,u\,dx=0
		\quad \text{for all } u\in \mathcal S(\Omega_0).
		\]
		This proves the $L^2$-density claim. Therefore, for every compact set $K\subset \Omega_0$ and every $\varepsilon>0$, there exists $v\in \mathcal S(\Omega)$ such that
		\[
		\|v-w\|_{L^2(K)}<\varepsilon,
		\]
		as desired.
	\end{proof}
	
	\begin{remark}
		{\rm Proposition~\ref{prop: weak Runge approximation}} is the weak form of Runge approximation needed below. The passage from $L^2$-approximation to $C^2$-approximation on compact subsets follows from interior elliptic regularity for homogeneous solutions.
	\end{remark}
	
	\begin{lemma}\label{lem: upgrade weak Runge to C2}
		Let $\Omega_1\Subset \Omega$ be a smooth subdomain. Suppose that $\{u_m\}$ is a sequence of solutions to $Lu_m=0$ in $\Omega_1$, and that $u_m\to 0 \quad \text{in } L^2(\Omega_1)$.
		Then for every compact set $K\subset \Omega_1$, $
		u_m\to 0$ in $C^2(K)$.
	\end{lemma}
	
	\begin{proof}
		Fix compact sets
		\[
		K\Subset \Omega_2\Subset \Omega_1,
		\]
		where $\Omega_2$ is a smooth subdomain. Since $Lu_m=0$ in $\Omega_1$, the interior Schauder estimate yields
		\[
		\|u_m\|_{C^{2,\alpha}(K)}
		\leq C_K \|u_m\|_{L^\infty(\Omega_2)},
		\]
		where $C_K>0$ is independent of $m$. On the other hand, by the interior estimate for uniformly elliptic equations, we have
		\[
		\|u_m\|_{L^\infty(\Omega_2)}
		\leq C_{\Omega_2}\|u_m\|_{L^2(\Omega_1)},
		\]
		for some constant $C_{\Omega_2}>0$ independent of $m$. Hence
		\[
		\|u_m\|_{C^{2,\alpha}(K)}
		\leq C \|u_m\|_{L^2(\Omega_1)},
		\]
		where $C>0$ is independent of $m$. Since $u_m\to 0$ in $L^2(\Omega_1)$, it follows that
		\[
		u_m\to 0
		\quad \text{in } C^{2,\alpha}(K),
		\]
		and in particular in $C^2(K)$.
	\end{proof}
	
	Combining Proposition~\ref{prop: spanning Hessians at a point}, Proposition~\ref{prop: weak Runge approximation}, and Lemma~\ref{lem: upgrade weak Runge to C2}, we obtain the required richness of global solutions.
	
	\begin{corollary}\label{cor: global spanning Hessians}
		Fix $x_0\in \Omega$, and let
		\[
		A(x_0)=\big(A_{k\ell}(x_0)\big)_{1\le k,\ell\le 2}
		:=\big(F_{k\ell}(D^2u_0^{(2)}(x_0))\big)_{1\le k,\ell\le 2}.
		\]
		Then there exist two global solutions $v_1,v_2\in H^2(\Omega)\cap C^{2,\alpha}_{\mathrm{loc}}(\Omega)$ of
		\[
		F_{k\ell}(D^2u_0^{(2)})\partial_{k\ell}v_j=0
		\quad \text{in } \Omega, \quad j=1,2,
		\]
		such that $D^2v_1(x_0)$ and $D^2v_2(x_0)$ are linearly independent elements of $K_{A(x_0)}$. Consequently,
		\[
		\mathrm{span}\big\{D^2v(x_0):\, F_{k\ell}(D^2u_0^{(2)})\partial_{k\ell}v=0 \text{ in } \Omega\big\}=K_{A(x_0)}.
		\]
	\end{corollary}
	
	\begin{proof}
		By Proposition~\ref{prop: spanning Hessians at a point}, there exist $r>0$ and local solutions $w_1,w_2\in C^{2,\alpha}(B_r(x_0))$ of
		\[
		F_{k\ell}(D^2u_0^{(2)})\partial_{k\ell}w_j=0
		\quad \text{in } B_r(x_0), \quad j=1,2,
		\]
		such that $D^2w_1(x_0)$ and $D^2w_2(x_0)$ are linearly independent in $K_{A(x_0)}$.
		
		Choose a smooth domain $\Omega_1$ such that $B_{r/2}(x_0)\Subset \Omega_1\Subset B_r(x_0)$. By Proposition~\ref{prop: weak Runge approximation}, for each $j=1,2$ there exists a sequence of global solutions $v_{j,m}\in H^2(\Omega)$ of
		\[
		F_{k\ell}(D^2u_0^{(2)})\partial_{k\ell}v_{j,m}=0
		\quad \text{in }\, \Omega
		\]
		such that $v_{j,m}\to w_j$ in $L^2(\Omega_1)$ as $m\to\infty$. Since each difference $v_{j,m}-w_j$ solves the same homogeneous equation in $\Omega_1$, Lemma~\ref{lem: upgrade weak Runge to C2} implies that
		\[
		v_{j,m}\to w_j
		\quad \text{in } C^2(\overline{B_{r/2}(x_0)})
		\]
		as $m\to\infty$. In particular,
		\[
		D^2v_{j,m}(x_0)\to D^2w_j(x_0)
		\quad \text{as } m\to \infty.
		\]
		
		Since linear independence is an open condition, it follows that for all sufficiently large $m$, the matrices $D^2v_{1,m}(x_0)$ and $D^2v_{2,m}(x_0)$ remain linearly independent. Moreover, since each $v_{j,m}$ satisfies the first linearized equation in $\Omega$, evaluating at $x_0$ yields
		\[
		A_{k\ell}(x_0)\partial_{k\ell}v_{j,m}(x_0)=0,
		\]
		and hence
		\[
		D^2v_{j,m}(x_0)\in K_{A(x_0)},
		\quad j=1,2.
		\]
		Therefore, for all sufficiently large $m$, the two matrices $D^2v_{1,m}(x_0)$ and $D^2v_{2,m}(x_0)$ are linearly independent elements of $K_{A(x_0)}$. Since $K_{A(x_0)}$ is two-dimensional, they form a basis of $K_{A(x_0)}$. Renaming
		\[
		v_j:=v_{j,m}, \quad j=1,2,
		\]
		for such a sufficiently large $m$, we obtain the desired conclusion.
	\end{proof}
	
	We now apply Corollary~\ref{cor: global spanning Hessians} to the identity \eqref{eq: gauge invariant derivative 4}. Fix an arbitrary $x_0\in \Omega$. Since \eqref{eq: gauge invariant derivative 4} holds for every global solution of the first linearized equation, Corollary~\ref{cor: global spanning Hessians} implies that, for each fixed pair $1\le a,b\le 2$, one has
	\[
	T^{ab}_{k\ell}(x_0)H_{k\ell}=0
	\quad \text{for all } H\in K_{A(x_0)},
	\]
	where
	\[
	T^{ab}_{k\ell}
	:=
	F_{ab,k\ell}(D^2u_0^{(1)})
	-
	c_0F_{ab,k\ell}(D^2u_0^{(2)}).
	\]
	Therefore, Lemma~\ref{lem: proportionality from orthogonality} yields a scalar function $\lambda^{ab}=\lambda^{ab}(x)$ such that
	\begin{equation}\label{eq: lambda-ab representation}
		F_{ab,k\ell}(D^2u_0^{(1)}(x))
		-
		c_0(x)F_{ab,k\ell}(D^2u_0^{(2)}(x))
		=
		\lambda^{ab}(x) F_{k\ell}(D^2u_0^{(2)}(x))
	\end{equation}
	for all $x\in \Omega$ and all $1\le a,b,k,\ell\le 2$.
	
	The next lemma shows that the coefficient $\lambda^{ab}$ is itself proportional to $A_{ab}=F_{ab}(D^2u_0^{(2)})$.
	
	\begin{lemma}\label{lem: lambda proportional to A}
		Suppose that
		\[
		T^{ab}_{k\ell}=\lambda^{ab}A_{k\ell}
		\]
		for some matrix-valued function $(\lambda^{ab})$, where $A=(A_{k\ell})$ satisfies the uniform ellipticity. Assume further that $T^{ab}_{k\ell}=T^{k\ell}_{ab}$. Then there exists a scalar function $\mu$ such that
		\[
		\lambda^{ab}=\mu A_{ab}.
		\]
		As a result, there holds $T^{ab}_{k\ell}=\mu A_{ab}A_{k\ell}$.
	\end{lemma}
	
	\begin{proof}
		From the identities
		\[
		T^{ab}_{k\ell}=\lambda^{ab}A_{k\ell}
		\quad \text{and} \quad
		T^{ab}_{k\ell}=T^{k\ell}_{ab},
		\]
		we obtain
		\[
		\lambda^{ab}A_{k\ell}=\lambda^{k\ell}A_{ab}
		\]
		for all $1\le a,b,k,\ell\le 2$. By the ellipticity, $A(x)$ is positive definite. Thus, $A_{11}(x)>0$ for any $x\in\Omega$.
		Define
		\[
		\mu(x):=\frac{\lambda^{11}(x)}{A_{11}(x)}.
		\]
		Then, taking $(k,\ell)=(1,1)$ in the above relation, we obtain
		\[
		\lambda^{ab}(x)A_{11}(x)=\lambda^{11}(x)A_{ab}(x),
		\]
		and hence $\lambda^{ab}(x)=\mu(x)A_{ab}(x)$. Substituting this back into $T^{ab}_{k\ell}=\lambda^{ab}A_{k\ell}$ gives $T^{ab}_{k\ell}=\mu A_{ab}A_{k\ell}$, as claimed.
	\end{proof}
	
	Therefore, applying Lemma~\ref{lem: lambda proportional to A} to \eqref{eq: lambda-ab representation} yields a scalar function $\mu=\mu(x)$ such that
	\begin{equation}\label{eq: mu-scalar representation}
		F_{ab,k\ell}(D^2u_0^{(1)}(x))
		-
		c_0(x)F_{ab,k\ell}(D^2u_0^{(2)}(x))
		=
		\mu(x) F_{ab}(D^2u_0^{(2)}(x))F_{k\ell}(D^2u_0^{(2)}(x))
	\end{equation}
	for all $x\in \Omega$ and all $1\le a,b,k,\ell\le 2$.
	
	\begin{remark}
		In this section, the homogeneity of $F$ and the injectivity of $DF$ are used only to show that the conformal factor does not vary with the boundary perturbation. The non-proportionality condition is not used here; it will be used only in Section~\ref{sec:inverse problem} to remove the remaining scalar factor.
	\end{remark}

	\section{The inverse problem}\label{sec:inverse problem}
	
	Finally, we can prove Theorem \ref{thm: main}.
	
	\begin{proof}[Proof of Theorem~\ref{thm: main}]
		By Proposition~\ref{prop: determination up to gauge}, there exists a positive function $c_0=c_0(x)$ such that $f_1=c_0^{\frac{m}{m-1}}f_2$ in $\Omega$. It remains to show that $c_0=1$.
		
		Let $u_0^{(j)}\in C^\infty(\overline{\Omega})$, $j=1,2$, be the admissible solutions to \eqref{eq: main background zero gauge}. From the proof of Proposition~\ref{prop: determination up to gauge}, we have
		\begin{equation}\label{eq: thm proof M1 tM2}
			D^2u_0^{(1)}=tD^2u_0^{(2)}
			\quad \text{in }\,\Omega,
		\end{equation}
		where $t=c_0^{1/(m-1)}$. Moreover, by the conclusion of Section~\ref{sec:second linearization}, in particular \eqref{eq: mu-scalar representation}, there exists a scalar function $\mu=\mu(x)$ such that
		\begin{equation}\label{eq: thm proof second derivative defect}
			F_{ab,k\ell}(D^2u_0^{(1)})
			-
			c_0F_{ab,k\ell}(D^2u_0^{(2)})
			=
			\mu F_{ab}(D^2u_0^{(2)})F_{k\ell}(D^2u_0^{(2)})
			\quad \text{in } \Omega.
		\end{equation}
		
		Fix any $x\in\Omega$ and set $M_2:=D^2u_0^{(2)}(x)$. Since $D^2u_0^{(1)}(x)=t(x)M_2$, and since the second derivative of $F$ is homogeneous of degree $m-2$, \eqref{eq: thm proof second derivative defect} gives
		\[
		t(x)^{m-2}F_{ab,k\ell}(M_2)-c_0(x)F_{ab,k\ell}(M_2)
		=
		\mu(x)F_{ab}(M_2)F_{k\ell}(M_2).
		\]
		Using $c_0(x)=t(x)^{m-1}$, this becomes
		\[
		t(x)^{m-2}(1-t(x))F_{ab,k\ell}(M_2)
		=
		\mu(x)F_{ab}(M_2)F_{k\ell}(M_2).
		\]
		If $t(x)^{m-2}(1-t(x))\neq 0$, then
		\[
		F_{ab,k\ell}(M_2)
		=
		\gamma F_{ab}(M_2)F_{k\ell}(M_2)
		\]
		for some $\gamma\in\mathbb R$, which contradicts Assumption~\ref{ass: homogeneous nondegeneracy}~\ref{item ass_non-prop}. Hence $t(x)^{m-2}(1-t(x))=0$. Since $t(x)>0$, we have $t(x)=1$, and therefore $c_0(x)=1$. Since $x\in\Omega$ was arbitrary, $c_0=1$ in $\Omega$. The gauge relation in Proposition~\ref{prop: determination up to gauge} then gives
		\[
		f_1=f_2
		\quad \text{in }\Omega.
		\]
		Together with \eqref{eq: BC for the source}, this proves $f_1=f_2$ in $\overline{\Omega}$.
	\end{proof}

	\section*{Statements and Declarations}
	
	\noindent\textbf{Data availability statement.}
	No datasets were generated or analyzed during the current study.
	
	\medskip
	
	\noindent\textbf{Conflict of Interests.} Hereby, we declare there are no conflicts of interest.

	\medskip

	\noindent\textbf{Acknowledgments.} 
	The authors would like to thank Mikko Salo for suggesting the inclusion of the gauge determination result.
	Y.-H. Lin is partially supported by the National Science and Technology Council (NSTC) of Taiwan, NSTC 113-2628-M-A49-003. Y.-H. Lin acknowledges financial support from the Alexander von Humboldt Foundation through the Henriette Herz Scouting Programme, hosted by Universität Duisburg-Essen. J.-N. Wang is partially supported by the National Science and Technology Council of Taiwan, NSTC 112-2115-M-002-010-MY3. The main results, arguments, and statements of this paper were developed without the use of AI. AI tools were used for auxiliary mathematical discussions and assistance with parts of the writing. The authors take full responsibility for the final mathematical statements and proofs.

	\bibliography{ref} 
	
	\bibliographystyle{alpha}

\end{document}